\documentclass[10pt,a4paper,reqno]{amsart}

\usepackage{amsmath,amsthm,amssymb}

\usepackage{setspace}
\usepackage{hyperref}
\hypersetup{%
    pdftitle={Quantum Limits of Eisenstein Series in H3},
    pdfauthor={Niko Laaksonen},
    colorlinks=true
}

\usepackage[backend=bibtex,bibencoding=ascii,style=numeric,firstinits=true]{biblatex}
\DeclareFieldFormat[article,inbook,incollection,inproceedings,patent,thesis,unpublished]{citetitle}{#1\isdot}
\DeclareFieldFormat[article,inbook,incollection,inproceedings,patent,thesis,unpublished]{title}{#1\isdot}

\renewcommand*{\intitlepunct}{\space}
\renewbibmacro{in:}{%
    \ifentrytype{article}{}{\printtext{\bibstring{in}\intitlepunct}}}
\DeclareFieldFormat[article]{volume}{\textbf{#1}}
\DeclareFieldFormat[article]{number}{\mkbibparens{#1}}
\renewbibmacro*{volume+number+eid}{%
    \printfield{volume}%
    \printfield{number}%
    \setunit{\addcomma\space}%
    \printfield{eid}}

\AtEveryBibitem{% Remove the following fields
    \clearfield{issn}
    \clearfield{url}
    \clearfield{bdsk-url-1}
    \clearfield{doi}
    \clearfield{isbn}
}
\DeclareNameAlias{sortname}{last-first}
\DeclareNameAlias{default}{last-first}

\usepackage[utf8]{inputenc}
\usepackage[T1]{fontenc}

\usepackage{mathrsfs} %for mathscr

\usepackage[margin=1in]{geometry}

\usepackage{enumerate}

\DeclareMathOperator{\re}{Re}
\DeclareMathOperator*{\res}{Res}
\DeclareMathOperator{\vol}{vol}

\newcommand{\naturals}{\mathbb{N}}
\newcommand{\integers}{\mathbb{Z}}
\newcommand{\rationals}{\mathbb{Q}}
\newcommand{\reals}{\mathbb{R}}

\newcommand{\complex}{\mathbb{C}}
\newcommand{\uph}{\mathbb{H}^{2}}

\newcommand{\uphs}{\mathbb{H}^{3}}

\newcommand{\hmodg}{\Gamma\backslash\uph}
\newcommand{\hmodgs}{\Gamma\backslash\uphs}

\newcommand{\tendsto}{\rightarrow}
\newcommand{\bigo}[2][]{\ensuremath{O_{#1}(#2)}}
\newcommand{\lilo}[2][]{o_{#1}(#2)}

\newcommand{\mellin}{\mathcal{M}}

\newcommand{\pslr}{\mathrm{PSL}_{2}(\mathbb{R})}
\newcommand{\pslz}{\mathrm{PSL}_{2}(\mathbb{Z})}
\newcommand{\slz}{\mathrm{SL}_{2}(\mathbb{Z})}

\newcommand{\pslc}{\mathrm{PSL}_{2}(\mathbb{C})}
\newcommand{\uphyp}{\mathbb{H}^{3}}
\newcommand{\ideala}{\mathfrak{a}}
\newcommand{\idealp}{\mathfrak{p}}
\newcommand{\ringo}{\mathcal{O}}
\newcommand{\ringok}{\mathcal{O}_{K}}
\newcommand{\dualo}{\mathcal{O}^{\ast}}
\newcommand{\units}{\abs{\ringo^{\times}}}

\newcommand{\gmodginf}{\Gamma_{\infty}\backslash\Gamma}
\newcommand{\dmu}{d\mu(p)}
\newcommand{\dxy}{\frac{dx_{1} dx_{2} dy}{y^{3}}}

\newcommand{\disc}{d_{K}}
\newcommand{\adisc}{\sqrt{\abs{\disc}}}

\newcommand{\modsum}{/\sim}
\newcommand{\volf}{\vol(F)}

\newcommand{\pslok}{\mathrm{PSL}_{2}(\ringo_{K})}
\newcommand{\inti}{\mathcal{I}}

\renewcommand{\d}[1]{\,d#1}

\newtheorem{thm}{Theorem}
\newtheorem{lemma}{Lemma}
\newtheorem{cor}{Corollary}

\theoremstyle{definition}

\newtheorem{remark}{Remark}

\usepackage{xparse} %for command with more than one optional argument
\usepackage{mathtools} %substack, mathclap and paireddelimiters

\DeclareDocumentCommand{\pd}{O{} O{} m}{\frac{\partial^{#1}#2}{\partial{#3}^{#1}}}

\let\abs\undefined
\DeclarePairedDelimiter{\abs}{\lvert}{\rvert}

\DeclarePairedDelimiterX{\inprod}[2]{\langle}{\rangle}{#1,#2}
\providecommand\given{}
\newcommand{\SetSymbol}{\nonscript\: :\nonscript\:\mathopen{}\allowbreak}
\DeclarePairedDelimiterX\Set[1]\{\}{%
    \renewcommand\given{\SetSymbol}
    #1
}

\usepackage{abstract}

\title{Quantum Limits of Eisenstein series in \texorpdfstring{$\mathbb{H}^{3}$}{H3}}
\author{Niko Laaksonen}

\thanks{The author was supported by the 150th Anniversary Postdoctoral Mobility Grant from the London Mathematical Society.}
\address{Department of Mathematics, University College London, Gower Street, London WC1E 6BT, United Kingdom}
\address{Department of Mathematical Sciences, University of Copenhagen, Universitetspark 5, 2100 Copenhagen \O, Denmark}
\email{n.laaksonen@ucl.ac.uk}

\date{\today}
\subjclass[2010]{Primary 11F72; Secondary 35P25}
\keywords{quantum limits; Eisenstein series; scattering poles; Bianchi groups}

\begin{document}
\maketitle
\begin{center}
    \begin{minipage}{.7\textwidth}
        \begin{abstract}
            \footnotesize
            We study the quantum limits of Eisenstein series off the critical line
            for $\pslok\backslash\uphs$, where $K$ is an imaginary quadratic field
            of class number one. This generalises the results of Petridis, Raulf
            and Risager on $\pslz\backslash\uph$. We observe that the measures
            $\abs{E(p,\sigma_{t}+it)}^{2}d\mu(p)$ become equidistributed only
            if $\sigma_{t}\tendsto 1$ as $t\tendsto\infty$. We use these computations to study
            measures defined in terms of the scattering states, which are shown to converge
            to the absolutely continuous measure $E(p,3)d\mu(p)$ under
            the GRH.
        \end{abstract}
    \end{minipage}
\end{center}

\section{Introduction}
Suppose $M$ is a compact negatively curved Riemannian manifold (without boundary) with the unit tangent bundle $X=SM$, then the geodesic flow
on $X$ is ergodic~\cite{sarnak2011}. The problem is to study the quantised flow,
in terms of the eigenfunctions $\phi_{j}$ of $\Delta$ on $M$, in the large eigenvalue limit. \textcite{shnirelman1974}, \textcite{zelditch1987}
and \textcite{colin-de-verdiere1985} proved that there is a full density subsequence of the measures $\mu_{j}=\abs{\phi_{j}}^{2}\mu$ which converges
weakly to the uniform measue $\mu$ on $M$.
It is not known in general whether $\mu$ is the unique limit.
When $M=\hmodg$, and $\Gamma$ is arithmetic, more tools are available such as Hecke operators and explicit Fourier expansions of Eisenstein series.
\textcite{rudnick1994} conjectured that for compact $M$ of constant negative curvature the limit $\mu$ is unique. This is the Quantum Unique Ergodicity (QUE)
conjecture.
For $\Gamma$ of arithmetic type the distribution of the eigenstates is well-understood.
In 1995 Luo and Sarnak~\cite{luo1995} proved the conjecture for Eisenstein series for non-compact arithmetic $\Gamma$ and, in particular,
for $\Gamma=\pslz$. The precise result is that given Jordan measurable subsets $A$ and $B$ of $M$, then
\begin{equation}\label{eq:luosarnak}
    \lim_{t\tendsto\infty}\frac{\int_{A}\abs{E(z,\frac{1}{2}+it)}^{2}\d{\mu(z)}}{\int_{B}\abs{E(z,\frac{1}{2}+it)}^{2}\d{\mu(z)}}=\frac{\mu(A)}{\mu(B)},
\end{equation}
where $\mu(B)\neq0$.
They actually compute the limit explicitly
\[\int_{A}\abs{E(z,\tfrac{1}{2}+it)}^{2}\d{\mu(z)}\sim\frac{6}{\pi}\mu(A)\log t,\]
as $t\tendsto\infty$ (the actual erroneous constant in~\cite{luo1995} is $48/\pi$, but it is not significant for their purposes).
Jakobson~\cite{jakobson1994} extended~\eqref{eq:luosarnak} to the unit tangent bundle. The result of Luo and Sarnak was also generalised
to $\pslok\backslash\uphyp$ by Koyama~\cite{koyama2000},
where $\pslok$ is the ring of integers of an imaginary quadratic field of class number one, and to
$\pslok\backslash\mathbb{H}^{n}$ with $K$ a totally real field of
degree $n$ and narrow class number one by Truelsen~\cite{truelsen2011}. In particular,
the quantum limit in~\cite{truelsen2011} for $\mu_{m,t}=\abs{E(z,\frac{1}{2}+it,m)}^{2}\mu$ is
\[\mu_{t,m}\tendsto\frac{(2\pi)^{n}nR}{2d_{K}\zeta_{K}(2)}\log t,\]
where $E(z,s,m)$ are a family of Eisenstein series parametrised by $m\in\integers^{n-1}$, $\zeta_{K}$ is the Dedekind zeta function and $R$ and $d_{K}$ are
the regulator and discriminant of $K$, respectively.
The QUE for $\phi_{j}$ a Hecke--Maa\ss~eigenform
was proven by Lindenstrauss~\cite{lindenstrauss2006} in the compact case and Soundararajan~\cite{soundararajan2010} in the non-compact case, thus
completing the full QUE conjecture for all arithmetic surfaces. \textcite{holowinsky2010} study QUE in the holomorphic case.
They consider holomorphic, $L^{2}$-normalised Hecke cusp forms $f_{k}$ of weight $k$ for $\slz$. They prove that the measures $\abs{y^{k/2}f_{k}(z)}^{2}\mu$ converge
weakly to $\mu$ as $k\tendsto\infty$.
Another interesting direction for the QUE of Eisenstein series has recently been proved by~\textcite{young2013},
who proves equidistribution of Eisenstein series for $\Gamma=\pslz$ when they are restricted to ``thin sets'', e.g.~geodesics connecting 0 and $\infty$ (as opposed
to restricting to compact Jordan measurable subsets of $\hmodg$ as in~\cite{luo1995}).
For a general cofinite $\Gamma\subset\pslr$ it is not clear whether there are infinitely many cusp forms so that
the limit of $\abs{\phi_{j}}^{2}\mu$ might not be relevant~\cites{phillips1985}{luo2001}.
\textcite{petridis2013a} (see also~\cite{petridis2013b}) propose to study the scattering states of $\Delta$ instead of the cuspidal spectrum.
It is known that under small deformations of $\Gamma$, the cusp forms dissolve into scattering states as characterised by Fermi's Golden Rule~\cites{phillips1992}{petridis2013c}.
The scattering states are described as residues of Eisenstein series on the left half-plane ($\re s<1/2$) at the non-physical poles
of the scattering matrix. These poles are called resonances. Let $\rho_{n}$ be a sequence of poles of the scattering matrix. For $\pslz\backslash\uph$ this
corresponds to half a non-trivial zero of $\zeta$.
Petridis, Raulf and Risager define the measures
\[u_{\rho_{n}}(z)=(\res_{s=\rho_{n}}\varphi(s))^{-1}\res_{s=\rho_{n}}E(z,s).\]
The normalisation is chosen so that $u_{\rho_{n}}$ has simple asymptotics $y^{1-\rho_{n}}$ for its growth at infinity.
The result is that for compact Jordan measurable subset $A$ of $\hmodg$,
\[\int_{A}\abs{u_{\rho_{n}}(z)}^{2}d\mu(z)\tendsto \int_{A}E(z,2-\gamma_{\infty})d\mu(z),\]
where $\gamma_{\infty}$ is the limit of the real part of the Riemann zeros. Under the RH the limit is $E(z,3/2)d\mu(z)$. This is obtained by studying the quantum limits
of Eisenstein series off the critical line.

We generalise their result to three dimensions $\hmodgs$ for $\Gamma$ a Bianchi group of class number one.
Let $\rho_{n}$ be a sequence of poles of the scattering matrix $\varphi(s)$ of $E(p,s)$ on $\hmodgs$ and define
\[\upsilon_{\rho_{n}}(p)=(\res_{s=\rho_{n}}\varphi(s))^{-1}\res_{s=\rho_{n}}E(p,s).\]
From the explicit form of $\varphi$~\eqref{eq:explicitk} we know that $\rho_{n}$ is equal to a
non-trivial zero of $\zeta_{K}$.
Define $s(t)=\sigma_{t}+it$, where $\sigma_{t}>1$ is a sequence converging to $\sigma_{\infty}\geq 1$.
Also, let $\gamma_{n}$ be the sequence of real parts of the non-trivial zeros of $\zeta_{K}$ with $\lim\gamma_{n}=\gamma_{\infty}$.
We will prove the following theorems.
\begin{thm}\label{thm:que1}
    Let $A$ be a compact Jordan measurable subset of $\hmodgs$. Then
    \[\int_{A}\abs{\upsilon_{\rho_{n}}(p)}^{2}\dmu\tendsto\int_{A}E(p,4-2\gamma_{\infty})\dmu\]
    as $n\tendsto\infty$.
\end{thm}
Notice that $4-2\gamma_{\infty}>2$ so that we are in the region of absolute convergence. Under the GRH the limit is $E(p,3)d\mu(p)$.
\begin{thm}\label{thm:que2}
    Assume $\sigma_{\infty}=1$ and $(\sigma_{t}-1)\log t\tendsto 0$. Let $A$ and $B$ be compact Jordan
    measurable subsets of $\hmodgs$. Then
    \[\frac{\mu_{s(t)}(A)}{\mu_{s(t)}(B)}\tendsto\frac{\mu(A)}{\mu(B)},\]
    as $t\tendsto\infty$. In fact, we have
    \begin{equation}\label{eq:quelimit}
        \mu_{s(t)}(A)\sim\mu(A)\frac{2(2\pi)^{2}}{\units\abs{\disc}\zeta_{K}(2)}\log t.
    \end{equation}
\end{thm}
Let $F$ be the fundamental domain of $\ringok$ as a lattice in $\reals^{2}$.
Since $\abs{F}=\sqrt{\abs{\disc}}/2$ and $\vol(\hmodgs)=\abs{\disc}^{3/2}\zeta_{K}(2)/(4\pi^{2})$, \cite[Proposition~2.1]{sarnak1983}, it is also possible to express the constant
in~\eqref{eq:quelimit} in terms of the volumes.
\begin{remark}\label{rem:mistake}
    The constant for the QUE of Eisenstein series \emph{on} the critical line in \textcite{koyama2000} is $2/\zeta_{K}(2)$.
    However, there is a small mistake in his computations on page~485, where the residue of the double pole of $\zeta_{K}^{2}(s/2)$
    goes missing. After fixing this (and taking into account the number of units of $\ringo$ which is normalised away in~\cite{koyama2000})
    his result agrees with our limit~\eqref{eq:quelimit} for $\sigma_{\infty}=1$.
\end{remark}
\begin{thm}\label{thm:que3}
    Assume $\sigma_{\infty}>1$. Let $A$ be a compact Jordan measurable subset of $\hmodgs$. Then
    \[\mu_{s(t)}(A)\tendsto\int_{A}E(p,2\sigma_{\infty})\dmu,\]
    as $t\tendsto\infty$.
\end{thm}
Theorem~\ref{thm:que1} says that the measures $\upsilon_{\rho}$ do \emph{not} become equidistributed. We could of course renormalise
the measures and use
\[d\nu_{s(t)}(p)=\abs*{\frac{E(p,s(t))}{\sqrt{E(p,2\sigma_{\infty})}}}^{2}\dmu.\]
Then we have the following corollary.
\begin{cor}\label{cor:que}
    Assume $\sigma_{\infty}>1$. Let $A$ be a compact Jordan measurable subset of $\hmodgs$. Then
    \[\nu_{s(t)}(A)\tendsto\mu(A),\]
    as $t\tendsto\infty$.
\end{cor}
The measures $\nu_{\rho_{n}}$ are not eigenfunctions of $\Delta$ so their equidistribution is not directly related to the QUE conjecture.
\begin{remark}
    \textcite{dyatlov2012} investigated quantum limits of Eisenstein series and scattering states for more general Riemannian manifolds with cuspidal ends.
    He proves results analogous to Theorems~\ref{thm:que1} and~\ref{thm:que3}.
    However, only the case of surfaces is explicitly written down and the limits are not identified as concretely for the arithmetic special cases such as
    in~\cite{petridis2013a} or Theorem~\ref{thm:que1} and~\ref{thm:que3}.
    Dyatlov uses a very different method of decomposing the Eisenstein series into plane waves and studying their microlocal limits, which does not use global
    properties of the surface, such as hyperbolicity.
\end{remark}

\section{Spectral Theory in \texorpdfstring{$\pslok\backslash\uphs$}{PSL(2,O)\textbackslash H}}
For the general spectral theory in hyperbolic three-space and the relevant facts about Bianchi groups and $\zeta_{K}$ we refer to~\cite{elstrodt1998}.
Fix a square-free integer $D<0$ and let
$K=\rationals(\sqrt{D})$ be the corresponding imaginary quadratic field of discriminant $\disc$.
Let $\ringo$ be the ring of integers of $K$ and let $\inprod{1}{\omega}$, where
\[\omega=\frac{\disc+\sqrt{\disc}}{2},\]
be a $\integers$-basis for $\ringo$.
Let $\Gamma=\mathrm{PSL}_{2}(\ringo)$.
For simplicity, restrict $D$ so that $K$ has class number one.
This means that $\Gamma$ has exactly one cusp (up to $\Gamma$-equivalence) which we may suppose is
$\infty\in\mathbb{P}^{1}\mathbb{C}$.
%The imaginary quadratic fields of class number one are exactly those with $D=-1$, $-2$, $-3$, $-7$, $-11$, $-19$, $-43$, $-67$, $-163$.
The Dedekind zeta function of $K$ is defined for $\re s>1$ by
\[\zeta_{K}(s)=\sideset{}{'}\sum_{\ideala\subset\ringo}\frac{1}{N\ideala^{s}}=\prod_{\idealp}\frac{1}{1-N\idealp^{-s}},\]
where the prime in the summation denotes that it is taken over
nonzero ideals $\ideala$, and the Euler product is taken over
prime ideals $\idealp\subset\ringo$. We define the completed zeta function by
\[\xi_{K}(s)=\left(\frac{\adisc}{2\pi}\right)^{s}\Gamma(s)\zeta_{K}(s).\]
Notice that this differs from the standard way of completing $\zeta_{K}$ due to the inclusion of the discriminant.
We know that $\xi_{K}$ satisfies a functional equation
\begin{equation}\label{eq:funceqk}
    \xi_{K}(s)=\xi_{K}(1-s),
\end{equation}
and has an analytic continuation to all of $\complex$ with a simple pole at $s=1$ with residue
\begin{equation}\label{eq:classnumber}
    \res_{s=1}\zeta_{K}(s)=\frac{2\pi}{\units},
\end{equation}
where $\units$ is the number of units of $\ringo$.

Let $\uphs=\Set{p=z+jy\given z\in\complex,\, y>0}$, where $z=x_{1}+ix_{2}$, be the three dimensional hyperbolic space.
The standard volume element on $\uphs$ is given by
\[\dmu=\frac{dx_{1}\,dx_{2}\,dy}{y^{3}},\]
and the hyperbolic Laplacian is
\[\Delta=y^{2}\left(\pd[2]{x_{1}}+\pd[2]{x_{2}}+\pd[2]{y}\right)-y\pd{y},\]
with the corresponding eigenvalue equation $\Delta f+\lambda f=0$. We write the eigenvalues of
$\Delta$ as $\lambda_{j}=s_{j}(2-s_{j})=1+t_{j}^{2}$. We know that for cofinite subgroups $\Gamma$ of $\pslc$ the Laplacian
has both discrete and continuous spectrum. The continuous spectrum spans (in the $\lambda$ aspect) the interval $[1,\infty)$ with
the eigenpacket given by Eisenstein series on the critical line, $E(p, 1+it)$. The discrete spectrum consists of Maa\ss~cusp forms and the eigenvalue $\lambda=0$.

The Eisenstein series for $\Gamma$ at the cusp at $\infty$ is given for $\re s>1$ by
\[E(p,s) = \sum_{\gamma\in\Gamma_{\infty}\backslash\Gamma}y(\gamma p)^{s},\]
where $\Gamma_{\infty}$ is the stabiliser of $\Gamma$ at $\infty$.
The Eisenstein series is an eigenfunction of $\Delta$ with the eigenvalue $\lambda=s(2-s)$, but it is not square integrable.
Let
\begin{equation}\label{eq:explicitk}
    \varphi(s) = \frac{\xi_{K}(s-1)}{\xi_{K}(s)}=\frac{2\pi}{\sqrt{\abs{\disc}}}\frac{1}{s-1}\frac{\zeta_{K}(s-1)}{\zeta_{K}(s)}
\end{equation}
be the scattering matrix of $E(p,s)$. The Fourier expansion of $E(p,s)$ at the cusp is then given by
\begin{equation}\label{eq:efourier3}
    E(p,s) = y^{s}+\varphi(s)y^{2-s}+\frac{2y}{\xi_{K}(s)}\sum_{0\neq
        n\in\ringo}\abs{n}^{s-1}\sigma_{1-s}(n)K_{s-1}\left(\frac{4\pi\abs{n}y}{\sqrt{\abs{\disc}}}\right)e^{2\pi
        i\inprod{\frac{2\overline{n}}{\sqrt{\disc}}}{z}},
\end{equation}
where
\[\sigma_{s}(n)=\sum_{\substack{(d)\subset\ringo\\ d|n}}\abs{d}^{2s}.\]
This form of the Fourier expansion can be found in~\cites{raulf2006,koyama2000}.
It also appears in a more general form in~\cites[(13)]{asai1970}[\S6~Theorem~2.11.]{elstrodt1998}.
The Eisenstein series $E(p,s)$ can be analytically continued to all of $\complex$ as a meromorphic function
of $s$. We can see from~\eqref{eq:explicitk} that to the right of the critical line $s=1$, $E(p,s)$ has only a simple pole at $s=2$ with residue
\[\res_{s=2}E(v,s)=\frac{\abs{\mathcal{F}_{\infty}}}{\vol(M)},\]
where $\mathcal{F}_{\infty}$ is the fundamental domain of $\Gamma_{\infty}$ acting on the boundary $\complex$, \cite[\S6~Theorem~1.11]{elstrodt1998}.
Moreover, since $\varphi(s)\varphi(2-s)=1$,
the Eisenstein series has a functional equation~\cite[\S6~Theorem~1.2]{elstrodt1998}
\begin{equation}\label{eq:efunceq3}
    E(p,s)=\varphi(s)E(p,2-s).
\end{equation}
We will also use the incomplete Eisenstein series which are defined for a smooth $\psi(x)$ with compact support on $\reals^{+}$ by
\[E(p|\psi)=\sum_{\gamma\in\Gamma_{\infty}\backslash\Gamma}\psi(y(\gamma p)).\]
As in two dimensions, it is possible to decompose $L^{2}(M)$ into the orthogonal spaces spanned by the closures of
the spaces of incomplete Eisenstein series on one hand, and the Maa\ss~cusp forms on the other hand.

Finally, for a Hecke--Maa\ss~cusp form $u_{j}$ we have the following Fourier expansion.
\begin{equation}\label{eq:cfourier3}
    u_{j}(p) = y\sum_{0\neq n\in\dualo_{K}}\rho_{j}(n)K_{it_{j}}(2\pi\abs{n}y)e^{2\pi i\inprod{n}{z}},
\end{equation}
where $\dualo_{K}$ is the dual lattice,
\[\dualo_{K}=\{ m : \inprod{m}{n}\in\integers\text{ for all }n\in\ringok\}.\]
Also, the Hecke eigenvalues satisfy $\rho_{j}(n) = \rho_{j}(1)\lambda_{j}(n)$, and in particular~\cite[Satz~16.8,~pg.~119]{heitkamp1992}
\begin{equation}\label{eq:lfact}
    L(u_{j},s)=\rho_{j}(1)\sum_{n\in\ringok}\frac{\lambda_{j}(n)}{N(n)^{s}}=\rho_{j}(1)\prod_{\idealp}(1-\lambda_{j}(\idealp)N\idealp^{-s}+N\idealp^{1-2s})^{-1}.
\end{equation}
We can split the space of Hecke--Maa\ss~cusp forms further into even and odd cusp forms depending on the sign in $\rho(-n)=\pm\rho(n)$.

\section{Proofs}
Let $M=\hmodgs$. Since any function in $L^{2}(M)$ can be decomposed in terms of the Hecke--Maa\ss~cusp forms $\Set{u_{j}}$
and the incomplete Eisenstein series $E(p|\psi)$, it is sufficient to consider them separately.

\subsection{Discrete Part}
We will first prove that the contribution of the discrete spectrum vanishes in the limit.
\begin{lemma}\label{lemma:disc}
    Let $u_{j}$ be a Hecke--Maa\ss~cusp form. Then
    \[\int_{M}u_{j}(p)\abs{E(p,s(t))}^{2}\dmu\tendsto0,\]
    as $t\tendsto\infty$.
\end{lemma}
\begin{proof}
    Denote the integral by
    \[J_{j}(s(t))=\int_{M}u_{j}(p)\abs{E(p,s(t))}^{2}\dmu.\]
    We define
    \[I_{j}(s)=\int_{M}u_{j}(p)E(p,s(t))E(p,s)\dmu.\]
    Unfolding the integral gives
    \[I_{j}(s)=\int_{0}^{\infty}\int_{F}u_{j}(p)E(p,s(t))y^{s}\dxy.\]
    After a change of variables, we may suppose that the $u_{j}$ in $I_{j}(s)$
    is even as the integral over the odd cusp forms vanishes.
    Substituting Fourier expansions of the Eisenstein series~\eqref{eq:efourier3} and the cusp forms~\eqref{eq:cfourier3}
    into the above integral gives
    \begin{multline*}
        I_{j}(s)=\int_{0}^{\infty}\int_{F}\biggl(2y\sum_{0\neq n\in\dualo}\rho_{j}(n)K_{ir_{j}}(2\pi\abs{n}y)\cos(2\pi\inprod{n}{z})\biggr)\\
        \times\biggl(y^{s(t)}+\varphi(s(t))y^{2-s(t)}\\
        +\frac{2y}{\xi_{K}(s(t))}\sum_{0\neq m\in\ringo}\abs{m}^{s(t)-1}\sigma_{1-s(t)}(m)K_{s(t)-1}\left(\frac{4\pi\abs{m} y}{\adisc}\right)e^{2\pi i\inprod{
                \frac{2\overline{m}}{\sqrt{\disc}}}{z}}\biggr)y^{s}\dxy.
    \end{multline*}
    By the definition of $F$ and the formula $\cos(a+b)=\cos a\cos b-\sin a\sin b$ it is simple to see that
    \[\int_{F}\cos(2\pi\inprod{n}{z})\d{z}=\begin{cases} 0, & \text{if } 0\neq n\in\dualo,\\ 1, & \text{if } n=0.\end{cases}\]
    Evaluating the integral over $F$ tells us that only the terms with $n=\pm 2m/\adisc$ remain and that the integral over the imaginary part goes to zero. Hence,
    with the identification $\ringo\rightarrow\dualo$ by $\alpha\mapsto(2/\sqrt{\disc})\overline{\alpha}$, we get
    \[I_{j}(s)=\frac{4}{\xi_{K}(s(t))}\int_{0}^{\infty}\sum_{0\neq n\in\dualo}\abs{n}^{s(t)-1}\sigma_{1-s(t)}(n)\rho_{j}(n)K_{s(t)-1}(2\pi\abs{n}y)K_{it_{j}}(2\pi\abs{n}y)y^{s}\frac{dy}{y}.\]
    The change of variables $y\mapsto y/\abs{n}$ yields
    \[I_{j}(s)=\frac{4}{\xi_{K}(s(t))}\sum_{0\neq n\in\dualo}\frac{\abs{n}^{s(t)-1}\sigma_{1-s(t)}(n)\rho_{j}(n)}{\abs{n}^{s}}\int_{0}^{\infty}K_{s(t)-1}(2\pi y)K_{it_{j}}(2\pi y)y^{s}\frac{dy}{y}.\]
    We can evaluate the integral by~\cite[6.576~(4)~and~9.100]{gradshteyn2007} to get
    \[I_{j}(s)= \frac{4}{\xi_{K}(s(t))}\frac{2^{-3}\pi^{-s}}{\Gamma(s)}\prod\Gamma\left(\frac{s\pm(s(t)-1)\pm it_{j}}{2}\right)R(s),\]
    where the product is taken over all combinations of $\pm$ and
    \[R(s)=\sum_{0\neq n\in\dualo}\frac{\abs{n}^{s(t)-1}\sigma_{1-s(t)}(n)\rho_{j}(n)}{\abs{n}^{s}}.\]
    Since $u_{j}$ is a Hecke eigenform, we can factorise $R(s)$ with~\eqref{eq:lfact} as
    \begin{align*}
        R(s) &= \rho_{j}(1)\prod_{(p):\text{prime ideal}}\sum_{k=0}^{\infty}\frac{\lambda_{j}(p^{k})\abs{p}^{k(s(t)-1)}\sigma_{1-s(t)}(p^{k})}{\abs{p}^{ks}}\\
        %&=\rho_{j}(1)\prod_{(p)}\sum_{k=0}^{\infty}\frac{\lambda_{j}(p^{k})\abs{p}^{k(s(t)-1)}}{\abs{p}^{ks}}\sum_{l=0}^{k}\abs{p}^{2(1-s(t))l}\\
        &=\rho_{j}(1)\prod_{(p)}\sum_{k=0}^{\infty}\frac{\lambda_{j}(p^{k})\abs{p}^{k(s(t)-1)}}{\abs{p}^{ks}}\frac{1-\abs{p}^{2(1-s(t))(k+1)}}{1-\abs{p}^{2(1-s(t))}},
    \end{align*}
        %&=\rho_{j}(1)\prod_{(p)}\frac{1}{(1-\abs{p}^{2(1-s(t))})}\biggl(\sum_{k=0}^{\infty}\lambda_{j}(p^{k})\abs{p}^{-k(s-s(t)+1)}\\
        %&\phantom{=}-\abs{p}^{-2(s(t)-1)}\sum_{k=0}^{\infty}\lambda_{j}(p^{k})\abs{p}^{-k(s+s(t)-1)}\biggr)\\
        %&=\rho_{j}(1)\prod_{(p)}\frac{1}{1-\abs{p}^{2(1-s(t))}}\biggl(\frac{1}{1-\lambda_{j}(p)\abs{p}^{-(s-s(t)+1)}+\abs{p}^{-2(s-s(t)+1)}}\\
        %&\phantom{=}-\frac{\abs{p}^{-2(s(t)-1)}}{1-\lambda_{j}(p)\abs{p}^{-(s+s(t)-1)}+\abs{p}^{-2(s+s(t)-1)}}\biggr)\\
    and thus
    \[R(s)=\rho_{j}(1)\prod_{(p)}\frac{1-\abs{p}^{-2s}}{1-\lambda_{j}(p)\abs{p}^{-(s-s(t)+1)}+\abs{p}^{-2(s-s(t)+1)}}
        \frac{1}{1-\lambda_{j}(p)\abs{p}^{-(s+s(t)-1)}+\abs{p}^{-2(s+s(t)-1)}}.\]
    We can identify the $L$-functions to get
    \[R(s)=\rho_{j}(1)\frac{L(u_{j},\frac{s-s(t)+1}{2})L(u_{j},\frac{s+s(t)-1}{2})}{\zeta_{K}(s)}.\]
    Now,
    \[J_{j}(t)=I_{j}(\overline{s(t)}),\]
    so that
    \begin{align*}
        J_{j}(t)&=\frac{2^{-1}}{\xi_{K}(s(t))}\frac{\pi^{-\overline{s(t)}}}{\Gamma(\overline{s(t)})}\prod\Gamma\left(\tfrac{\overline{s(t)}\pm(s(t)-1)\pm
                it_{j}}{2}\right)\rho_{j}(1)\frac{1}{\zeta_{K}(\overline{s(t)})}L(u_{j},\tfrac{1}{2}-it)L(u_{j},\sigma_{t}-\tfrac{1}{2})\\
        &=\frac{2^{s(t)-1}\pi^{2it}\rho_{j}(1)}{\abs{\disc}^{s(t)/2}\abs{\zeta_{K}(s(t))}^{2}}L(u_{j},\tfrac{1}{2}-it)L(u_{j},\sigma_{t}-\tfrac{1}{2})\frac{\prod\Gamma\left(\frac{\overline{s(t)}\pm(s(t)-1)\pm
                    it_{j}}{2}\right)}{\abs{\Gamma(s(t))}^{2}}.
    \end{align*}
    With Stirling asymptotics we see that the quotient of Gamma factors is $\bigo{\abs{t}^{1-2\sigma_{t}}}$.
    We use the estimate
    \begin{equation}\label{eq:zetaklog}
        \log^{-2}\abs{t}\ll \zeta_{K}(s(t))\ll\log^{2}\abs{t},
    \end{equation}
    which follows by adapting~\cite[(3.5.1)~and~Theorem~3.11]{titchmarsh1986} for $L(s,\chi)$ and the zero-free region~\cite{landau1924a}.
    For the $L$-functions we need a subconvex bound to guarantee vanishing.
    \textcite{petridis2001} show that there is a $\delta>0$ such that
    \[ L(u_{j},\tfrac{1}{2}+it)\ll_{j}\abs{1+t}^{1-\delta}.\]
    In fact, they have $\delta=7/166$, although this is not crucial for us.
    Hence, $J_{j}(t)\tendsto 0$, as $t\tendsto\infty$.
\end{proof}

\subsection{Continuous Part}
Let $h(y)\in C^{\infty}(\reals^{+})$ be a rapidly decreasing function at 0 and $\infty$ so that
$h(y)=\bigo[N]{y^{N}}$ for $0<y<1$ and $h(y)=\bigo[N]{y^{-N}}$ for $y\gg 1$ for all $N\in\naturals$. Denote the Mellin transform of $h$ by $H=\mellin h$, i.e.
\[H(s)=\int_{0}^{\infty}h(y)y^{-s}\frac{dy}{y}\]
and the Mellin inversion formula gives
\[h(y)=\frac{1}{2\pi i}\int_{(\sigma)}H(s)y^{s}\d{s},\]
for any $\sigma\in\reals$. We consider the incomplete Eisenstein series denoted by
\[F_{h}(p)=E(p|h)=\sum_{\gamma\in\gmodginf}h(y(\gamma p))=\frac{1}{2\pi i}\int_{(3)}H(s)E(p,s)\d{s},\]
where $h$ is a smooth function on $\reals^{+}$ with compact support.
We prove the following lemma.
\begin{lemma}\label{lemma:cont}
    Let $h$ be a function satisfying the conditions stated above. Then
    \begin{equation*}
        \int_{M}F_{h}(p)\abs{E(v,s(t))}^{2}\dmu\sim
        \begin{cases}
            \int_{M}F_{h}(p)E(p,2\sigma_{\infty})\dmu, & \text{if $\sigma_{\infty}>1$,}\\
            \frac{2(2\pi)^{2}}{\units\abs{\disc}\zeta_{K}(2)}\log t \int_{M}F_{h}(p)\dmu, & \text{if $(\sigma_{t}-1)\log t\tendsto 0,$}
        \end{cases}
    \end{equation*}
    as $t\tendsto\infty$.
\end{lemma}
Now, unfolding gives
\begin{align*}
    \int_{M}F_{h}(p)\abs{E(p,s(t))}^{2}\dmu &= \int_{M}\frac{1}{2\pi i}\int_{(3)}H(s)E(p,s)\d{s}\,\abs{E(p,s(t))}^{2}\dmu\\
    &= \int_{0}^{\infty}\frac{1}{2\pi i}\int_{(3)}H(s)y^{s}\d{s}\,\int_{F}\abs{E(p,s(t))}^{2}\dmu\\
    &=\int_{0}^{\infty} h(y)\volf\left(\sum_{n\in\ringo}\abs{a_{n}(y,s(t))}^{2}\right)\frac{dy}{y^{3}}.
\end{align*}
We will deal separately with the contribution of the $n=0$ term and the rest. We factor out the constant $\volf$ in the analysis below.

\subsubsection{Contribution of the constant term}
We know that
\[\abs{a_{0}(y,s(t))}^{2} = y^{2\sigma_{t}}+2\re(\varphi(s(t))y^{2-2it})+\abs{\varphi(s(t))}^{2}y^{4-2\sigma_{t}}.\]
The first term is
\[\int_{0}^{\infty}h(y)y^{2\sigma_{t}-2}\frac{dy}{y}=H(2-2\sigma_{t}),\]
which converges to $H(2-2\sigma_{\infty})$. For the second term we first have that
\[\varphi(s(t))\int_{0}^{\infty}h(y)y^{-2it}\frac{dy}{y}=\varphi(s(t))H(2it).\]
Since $H(s)$ is in Schwartz class in $t$, the function $H(2it)$ decays rapidly, whereas $\varphi(s(t))$ is bounded.
By taking complex conjugates we see that the second term will also tend to zero.
Finally, for the third expression in the constant term we get
\[\abs{\varphi(s(t))}^{2}\int_{0}^{\infty}h(y)y^{2-2\sigma_{t}}\frac{dy}{y}=\abs{\varphi(s(t))}^{2}H(2\sigma_{t}-2).\]
If $\sigma_{\infty}\neq 1$ then
\[\abs{\varphi(s(t))}=\left\lvert\frac{2\pi}{s(t)-1}\frac{\zeta_{K}(s(t)-1)}{\zeta_{K}(s(t))}.\right\rvert\]
To estimate this we need the convexity bound for $\zeta_{K}$,
\[\zeta_{K}(s(t)-1)=\zeta_{K}(\sigma_{t}-1+it)=\bigo{\abs{t}^{1-\sigma_{t}/2+\epsilon}},\]
and of course
\[\frac{1}{s(t)-1}=\bigo{\abs{t}^{-1}}.\]
Combining all of this with~\eqref{eq:zetaklog} we get
\[\varphi(s(t))=\bigo{\abs{t}^{-\sigma_{t}/2+\epsilon}},\]
and so
\begin{equation}\label{eq:phizero}
    \varphi(s(t))\tendsto 0,
\end{equation}
as $t\tendsto\infty$, when $\sigma_{\infty}\neq 1$. So in summary, the contribution of the constant term converges to $H(2-2\sigma_{\infty})$ if $\sigma_{\infty}\neq 1$ and is $\bigo{1}$ otherwise.

\subsubsection{Contribution of the non-constant terms}
In this case the contribution equals
\begin{align*}
    A(t) &=\int_{0}^{\infty}\frac{1}{2\pi i}\int_{(3)}H(s)y^{s}\d{s}\frac{4y^{2}}{\abs{\xi_{K}(s(t))}^{2}}\sideset{}{'}\sum_{n\in\ringo}\abs{n}^{2\sigma_{t}-2}\abs{\sigma_{1-s(t)}(n)}^{2}\abs*{K_{s(t)-1}\left(\tfrac{4\pi\abs{n}y}{\adisc}\right)}^{2}\frac{dy}{y^{3}}\\
    &=\frac{4\units}{\abs{\xi_{K}(s(t))}^{2}}\frac{1}{2\pi
        i}\int_{(3)}H(s)\left(\tfrac{\adisc}{4\pi}\right)^{s}\sideset{}{'}\sum_{n\in\ringo\modsum}\frac{\abs{\sigma_{1-s(t)}(n)}^{2}}{\abs{n}^{s+2-2\sigma_{t}}}\int_{0}^{\infty}y^{s}\abs{K_{s(t)-1}(y)}^{2}\,\frac{dy}{y}\d{s},
\end{align*}
where $a\sim b$ if $a$ and $b$ generate the same ideal in $\ringo$ and prime in the summation denotes that it is taken over $n\neq0$.
We now need to evaluate the series. Keeping in mind that $N(p)=\abs{p}^{2}$, we get by a standard calculation
\begin{align*}
    \sideset{}{'}\sum_{n\in\ringo\modsum}\frac{\sigma_{a}(n)\sigma_{b}(n)}{\abs{n}^{s}}
    &=\prod_{(p):\text{prime ideal}}\sum_{k=0}^{\infty}\frac{\sigma_{a}(p^{k})\sigma_{b}(p^{k})}{\abs{p}^{ks}}\\
    %&=\prod_{(p)}\sum_{k=0}^{\infty}\frac{1}{\abs{p}^{ks}}\left(\frac{1-\abs{p}^{2a(k+1)}}{1-\abs{p}^{2a}}\right)\left(\frac{1-\abs{p}^{2b(k+1)}}{1-\abs{p}^{2b}}\right)\\
    %&=\prod_{(p)}\frac{1}{1-\abs{p}^{2a}}\frac{1}{1-\abs{p}^{2b}}\sum_{k=0}^{\infty}\biggl(\abs{p}^{-ks}\\
    %&\phantom{=}-\abs{p}^{k(2a-s)+2a}-\abs{p}^{k(2b-s)+2b}+\abs{p}^{k(2a+2b-s)+2a+2b}\biggr),\\
    %\shortintertext{and hence}
    %\sideset{}{'}\sum_{n\in\ringo\modsum}\frac{\sigma_{a}(n)\sigma_{b}(n)}{\abs{n}^{s}}
    %&=\prod_{(p)}\frac{1}{1-\abs{p}^{2a}}\frac{1}{1-\abs{p}^{2b}}\Biggl(\frac{1}{1-\abs{p}^{-ks}}-\frac{\abs{p}^{2a}}{1-\abs{p}^{2a-s}}\\
    %&\phantom{=}-\frac{\abs{p}^{2b}}{1-\abs{p}^{2b-s}}+\frac{\abs{p}^{2a+2b}}{1-\abs{p}^{2a+2b-s}}\Biggr)\\
    &=\prod_{(p)}\frac{1-\abs{p}^{2(a+b-s)}}{(1-\abs{p}^{-s})(1-\abs{p}^{2a-s})(1-\abs{p}^{2b-s})(1-\abs{p}^{2a+2b-s})},
\end{align*}
and hence
\[\sideset{}{'}\sum_{n\in\ringo\modsum}\frac{\sigma_{a}(n)\sigma_{b}(n)}{\abs{n}^{s}}=\frac{\zeta_{K}(\frac{s}{2})\zeta_{K}(\frac{s}{2}-a)\zeta_{K}(\frac{s}{2}-b)\zeta_{K}(\frac{s}{2}-a-b)}{\zeta_{K}(s-a-b)}.\]
For $a=\overline{b}=1-s(t)$ and $s=s-2(\sigma_{t}-1)$ this becomes
\[\sideset{}{'}\sum_{n\in\ringo\modsum}\frac{\abs{\sigma_{1-s(t)}(n)}^{2}}{\abs{n}^{s-2\sigma_{t}+2}} =
    \frac{\zeta_{K}(\frac{s}{2}-\sigma_{t}+1)\zeta_{K}(\frac{s}{2}+it)\zeta_{K}(\frac{s}{2}-it)\zeta_{K}(\frac{s}{2}+\sigma_{t}-1)}{\zeta_{K}(s)}.\]
Again, by~\cite[6.576~(4)]{gradshteyn2007} we see that
\[\int_{0}^{\infty}y^{s}\abs{K_{s(t)-1}(y)}^{2}\frac{dy}{y}=\frac{2^{s-3}}{\Gamma(s)}\Gamma(\tfrac{s}{2}-\sigma_{t}+1)\Gamma(\tfrac{s}{2}+it)\Gamma(\tfrac{s}{2}-it)\Gamma(\tfrac{s}{2}+\sigma_{t}-1).\]
Hence, $A(t)$ becomes
\begin{align*}
    A(t)&=\frac{\units}{\abs{\xi_{K}(s(t))}^{2}}\frac{1}{4\pi
        i}\int_{(3)}H(s)\frac{\xi_{K}(\frac{s}{2}-\sigma_{t}+1)\xi_{K}(\frac{s}{2}-it)\xi_{K}(\frac{s}{2}+it)\xi_{K}(\frac{s}{2}+\sigma_{t}-1)}{\xi_{K}(s)}\d{s}\\
    &=\frac{\units}{\abs{\xi_{K}(s(t))}^{2}}\frac{1}{4\pi i}\int_{(3)}B(s)\d{s},
\end{align*}
say.
By the Dirichlet Class Number Formula~\eqref{eq:classnumber} for $\zeta_{K}$, the completed zeta function $\xi_{K}$ has a simple pole at $s=1$ with
\[\res_{s=1}\xi_{K}(s)=\frac{1}{\units}.\]
There is also a simple pole at $s=0$.
%, \cite[Theorem~10.5.1~(3)]{cohen2007b}.
It follows that the poles of $B(s)$ in the region
$\re s\geq 1$ are at $2\pm 2it$, $2\sigma_{t}$, $2\sigma_{t}-2$, and $4-2\sigma_{t}$.
Moving the line of integration to $\re s=1$ gives
\begin{align*}
    A(t)&=\frac{\units}{2\abs{\xi_{K}(s(t))}^{2}}\Biggl(\res_{s=2\pm
        2it}B(s)+\res_{s=2\sigma_{t}}B(s)+\delta_{t}\res_{s=4-2\sigma_{t}}B(s)\\
    &\phantom{=}+(1-\delta_{t})\res_{s=2\sigma_{t}-2}B(s)+\frac{1}{2\pi i}\int_{(1)}B(s)\d{s}\Biggr),\\
    &= A_{1}+A_{2}+\cdots+A_{5},
\end{align*}
where $\delta_{t}=1$ if $\sigma_{t}<3/2$ and 0 otherwise. We deal with each of the residues $A_{i}$ separately.
    For the first term we have
        \[A_{1}=\frac{H(2\pm 2it)}{\abs{\xi_{K}(\sigma_{t}+it)}^{2}}\frac{\xi_{K}(2-\sigma_{t}\pm it)\xi_{K}(1\pm 2it)\xi_{K}(\sigma_{t}\pm it)}{\xi_{K}(2\pm 2it)}.\]
        By Stirling asymptotics and convexity estimates for the Dedekind zeta functions, the quotient of the $\xi_{K}$ functions
        is bounded by $\abs{t}^{1-2\sigma_{t}}\log^{10}\abs{t}$. By virtue of $H$ being of rapid decay in $t$
        it follows that $A_{1}\tendsto 0$ as $t\tendsto\infty$.

    The second term is
        \[A_{2}=H(2\sigma_{t})\frac{\xi_{K}(2\sigma_{t}-1)}{\xi_{K}(2\sigma_{t})}.\]
        If $\sigma_{\infty}\neq 1$ then
        \[A_{2}\tendsto H(2\sigma_{\infty})\frac{\xi_{K}(2\sigma_{\infty}-1)}{\xi_{K}(2\sigma_{\infty})},\]
        but if $\sigma_{t}\tendsto 1$ then
        \[ A_{2}\sim H(2)\frac{1}{2\units\xi_{K}(2)(\sigma_{t}-1)}.\]

    Now, in the third term we use the form~\eqref{eq:explicitk} of $\varphi$ and the fact that
        $\xi_{K}$ satisfies the functional equation~\eqref{eq:funceqk}
        \[\xi_{K}(s)=\xi_{K}(1-s).\]
        We can then write
        \[A_{3}=\delta_{t}H(4-2\sigma_{t})\abs{\varphi(s(t))}^{2}\frac{\xi_{K}(3-2\sigma_{t})}{\xi_{K}(4-2\sigma_{t})}.\]
        By~\eqref{eq:phizero} we have that $\varphi(s(t))\tendsto 0$ as $t\tendsto\infty$ for $\sigma_{\infty}\neq 1$. Hence, if $\sigma_{\infty}\neq 1$, then
        \[A_{3}\tendsto 0.\]
        On the other hand, if $\sigma_{\infty}=1$, then
        \[A_{3}\sim\frac{\delta_{t}}{2\units}H(2)\abs{\varphi(s(t))}^{2}\frac{-1}{\xi_{K}(2)(\sigma_{t}-1)},\]
        which is bounded.

    For the fourth term we have
        \[A_{4}=(1-\delta_{t})\zeta_{K}(0)H(2\sigma_{t}-2)\abs{\varphi(s(t))}^{2},\]
        which clearly converges to 0 if $\sigma_{\infty}\neq 1$ and is bounded for $\sigma_{\infty}=1$ as
        in the previous case.

    Finally, the fifth term is
        \begin{align*}
            A_{5} &=\frac{\units}{2\abs{\xi_{K}(s(t))}^{2}}\frac{1}{2\pi i}\int_{(1)}B(s)\d{s}=\frac{\units}{2\abs{\xi_{K}(\sigma_{t}+it)}^{2}}\inti,
        \end{align*}
        where
        \[\inti=\frac{1}{2\pi}\int_{-\infty}^{\infty}H(1+i\tau)\frac{\abs{\xi_{K}(\sigma_{t}-\frac{1}{2}+i\tau)}^{2}\xi_{K}(\frac{1}{2}+i(\tau+t))\xi_{K}(\frac{1}{2}+i(\tau-t))}{\xi_{K}(1+2i\tau)}\d{\tau}.\]
        We now estimate the growth of $A_{5}$ in terms of $t$.
        The exponential contribution from the gamma functions in the integral is equal to
        \[(e^{-\frac{\pi}{2}\abs{t}})^{2}e^{-\frac{\pi}{2}\abs{\tau+t}}e^{-\frac{\pi}{2}\abs{\tau-t}}e^{\frac{\pi}{2}\abs{2\tau}}\ll e^{-\pi\abs{t}}.\]
        This cancels with the exponential growth of $\abs{\xi_{K}(s(t))}^{2}$.
        Since $H(1+i\tau)$ decays rapidly, we can bound $\zeta_{K}(\sigma_{t}-\frac{1}{2}+i\tau)$ polynomially and absorb it
        into $H$. Hence
        \begin{align*}
            A_{5} &\ll
            \frac{\log^{4}\abs{t}}{(\abs{t}^{\sigma_{t}-1/2})^{2}}\int_{-\infty}^{\infty}\widetilde{H}(\tau)\abs{\zeta_{K}(\tfrac{1}{2}+i(\tau+t))\zeta_{K}(\tfrac{1}{2}+i(\tau-t))}\d{\tau}.
        \end{align*}
        where $\widetilde{H}$ is a function of rapid decay. The Dedekind zeta functions can be estimated with the subconvex bound
        \[\zeta_{K}(\tfrac{1}{2}+it)\ll t^{1/3+\epsilon}\]
        due to~\textcite{heath-brown1988} (he proves a more general bound for arbitrary number fields of degree $n$).
        We get
        \begin{equation*}
            A_{5}\ll \abs{t}^{5/3-2\sigma_{t}+2\epsilon}\log^{4}\abs{t}\int_{-\infty}^{\infty}\widetilde{H}(\tau)
            (t^{-1}+\abs{\tau t^{-1}+1})^{1/3+\epsilon}(t^{-1}+\abs{\tau t^{-1}-1})^{1/3+\epsilon}\d{\tau},
        \end{equation*}
        which is $\lilo{1}$ since $\sigma_{t}\geq 1$.

Hence we have proved that the integral
\[\int_{M}F_{h}(p)\abs{E(p,s(t))}^{2}\dmu\]
converges to
\[\volf\left(H(2-2\sigma_{\infty})+H(2\sigma_{\infty})\frac{\xi_{K}(2\sigma_{\infty}-1)}{\xi_{K}(2\sigma_{\infty})}\right),\]
if $\sigma_{\infty}\neq 1$. On the other hand, for $\sigma_{\infty}=1$ the contribution is asymptotic to
\begin{equation}\label{eq:finalestimate2}
    \volf H(2)\frac{1-\abs{\varphi(s(t))}^{2}}{2\units\xi_{K}(2)(\sigma_{t}-1)}+O(1).
\end{equation}
To finish the proof, we apply Mellin inversion and unfold backwards to see that
\begin{align*}
    \volf \biggl(H(2-2\sigma_{\infty})+&H(2\sigma_{\infty})\frac{\xi_{K}(2\sigma_{\infty}-1)}{\xi_{K}(2\sigma_{\infty})}\biggr)\\
    &=\int_{0}^{\infty}h(y)\volf(y^{2\sigma_{\infty}-2+2}+\varphi(2\sigma_{\infty})y^{-2\sigma_{\infty}+2})\frac{dy}{y^{3}}\\
    &=\int_{0}^{\infty}h(y)\left(\int_{F}E(z+jy,2\sigma_{\infty})\d{z}\right)\frac{dy}{y^{3}}\\
    &=\int_{M}F_{h}(p)E(p,2\sigma_{\infty})\dmu,
\end{align*}
and
\[\volf H(2)=\int_{M}F_{h}(p)\dmu.\]
For the second case, we need to estimate the quotient with the scattering matrix. We will show that
\begin{align*}
    \frac{1-\abs{\varphi(s(t))}^{2}}{2\units\xi_{K}(2)(\sigma_{t}-1)}&\sim\frac{2(2\pi)^{2}}{\units\abs{\disc}\zeta_{K}(2)}\log t.
\end{align*}
Let $G(\sigma)=\varphi(\sigma+it)\varphi(\sigma-it)$ and notice that
\[G'(\sigma)=\frac{\varphi'}{\varphi}(\sigma\pm it)G(\sigma),\]
where the $\pm$ denotes the linear combination
$\frac{\varphi'}{\varphi}(\sigma\pm it)=\frac{\varphi'}{\varphi}(\sigma+it)+\frac{\varphi'}{\varphi}(\sigma-it)$.
We then apply the mean value theorem twice on the intervals $[1,\sigma]$ and $[1,\sigma']$, respectively. We get
\begin{align*}
    \frac{G(1)-G(\sigma)}{1-\sigma}%&=G'(\sigma')\\
    %&=G(\sigma')\frac{\varphi'}{\varphi}(\sigma'\pm it)\\
    &=\left(G(1)-(1-\sigma')G(\sigma'')\frac{\varphi'}{\varphi}(\sigma''\pm it)\right)\frac{\varphi'}{\varphi}(\sigma'\pm it),
\end{align*}
where $1\leq\sigma''\leq\sigma'\leq\sigma$.
On noticing that $G(1)=1$, this gives
\[\frac{1-\abs{\varphi(\sigma+it)}^{2}}{1-\sigma}=\left(1-(1-\sigma')\abs{\varphi(\sigma''+it)}^{2}\frac{\varphi'}{\varphi}(\sigma''\pm it)\right)\frac{\varphi'}{\varphi}(\sigma'\pm it).\]
Using the asymptotics
\begin{equation}\label{eq:phiasymp}
    \frac{\varphi'}{\varphi}(\sigma\pm it)\sim-4\log t,
\end{equation}
and the fact that $\abs{\varphi(\sigma+it)}$ is bounded for $\sigma\geq 1$ proves the lemma.
The estimate~\eqref{eq:phiasymp} follows immediately from the standard asymptotics for the digamma function,
\[\frac{\Gamma'}{\Gamma}(\sigma+it) = \log\abs{t} + \bigo{1},\]
and the Weyl bound
\[\frac{\zeta_{K}'}{\zeta_{K}}(\sigma+it)\ll\frac{\log t}{\log\log t},\]
for $\zeta_{K}'/\zeta_{K}$ (see~\cite[Theorems~3.11~and~5.17]{titchmarsh1986} and~\cite{coleman1990}). \hfill$\qed$
\begin{proof}[Proofs of Theorems~\ref{thm:que2}~and~\ref{thm:que3}]
    These follow now from Lemmas~\ref{lemma:disc}~and~\ref{lemma:cont} by approximation arguments similar to~\cite{luo1995} and~\cite{koyama2000}.
\end{proof}
Theorem~\ref{thm:que1} now follows easily.
\begin{proof}[Proof of Theorem~\ref{thm:que1}]
    By the functional equation~\eqref{eq:efunceq3} of $E(p,s)$, we get
    \begin{align*}
        \abs{\upsilon_{\rho_{n}}}^{2}\dmu &= \abs{(\res_{s=\rho_{n}}\varphi(s))^{-1}\res_{s=\rho_{n}}E(p,s)}^{2}\dmu\\
        &= \abs{(\res_{s=\rho_{n}}\varphi(s))^{-1}\res_{s=\rho_{n}}\varphi(s)E(p,2-s)}^{2}\dmu\\
        &=\abs{E(p,2-\rho_{n})}^{2}\dmu.
    \end{align*}
    We apply Theorem~\ref{thm:que3} with $\sigma_{\infty}=2-\gamma_{\infty}$ to conclude the proof.
\end{proof}

\printbibliography

\end{document}